\documentclass[a4paper,12pt]{amsart}
\pdfoutput=1
\usepackage[protrusion=true,expansion=true]{microtype}
\usepackage[T1]{fontenc} 
\usepackage{mathtools}
\usepackage[all]{xy}
\usepackage[english]{babel}
\usepackage{frcursive} 
\usepackage{eucal}
\usepackage{graphicx}
\usepackage{epsfig}
\usepackage{mathrsfs}
\usepackage{amssymb}
\usepackage{amsxtra}
\usepackage{enumerate}

\newtheorem{theorem}{Theorem}[section]

\newtheorem{lemma}[theorem]{Lemma}

\newtheorem{definition}[theorem]{Definition}

\theoremstyle{remark}

\def \1{\mathbb {1}}
\def \RM{\mathbb {R}}
\def \NM{\mathbb{N}}
\def \ZM{\mathbb{Z}}
\def \CM{\mathbb{C}}

\def \d{\partial}
\def\dt{\delta} 
\def\a{\alpha}
\def\b{\beta}
\def\e{\varepsilon}  
\def\g{\gamma}

\def\G{\Gamma}

\def \to{\longrightarrow} 

\def \< {{\langle }}
\def \> {{\rangle }}
\def \( {\left( }
\def \) {\right) }

\newcommand{\Lt}{{\mathcal L}}

\newcommand{\Ot}{{\mathcal O}}

\newcommand{\Ut}{{\mathcal U}}
\newcommand{\Vt}{{\mathcal V}}

\newcommand{\lra}{\longrightarrow}
\renewcommand{\S}{\Sigma}
\newcommand{\la}{\langle}
\newcommand{\ra}{\rangle}
\renewcommand{\o}{\omega}

\title[Non-degenerate Birkhoff functions]{Non-degenerate Birkhoff functions }
\author{  Mauricio  Garay and Duco van Straten}

\begin{document}
\begin{abstract}  It is known after the works of Mahler, Kleinbock, Margulis, Sprind\v{z}uk that very well approximated numbers on a manifold form a zero measure set, assuming non-degeneracy conditions.  These non-degeneracy conditions are,
  in many applications, difficult to check. We propose here a setting in which these are automatic. Rather than functions, we consider correspondences which have a better behaviour.
\end{abstract}
\maketitle

\section{Introduction}
In 1942, Siegel proved that a biholomorphic germ
$$f:(\CM,0) \to (\CM,0), \ z \mapsto \a z +o(|z|) $$
can be linearised, assuming an arithmetic condition on the coefficient
$\a$~\cite{Siegel_linearisation}. Ever since, the field of Diophantine
approximation has been of relevance for dynamical systems and this led
Arnold, Margulis and Pyartli and others to reconsider the classical works
on Diophantine approximation and extend their scope~\cite{Kleinbock_Margulis,Pyartli}. Nowadays, the work of Kleinbock and Margulis finds application
in KAM theory in order to prove that in certain situations the quasiperiodic
solutions make up a set of positive measure in phase space~\cite{Russmann_KAM}.
To apply these results, one needs to assume certain {\em non-degeneracy conditions}, introduced in the works of Mahler, Sprind\v{z}uk, Margulis and Kleinbock~\cite{Kleinbock_Margulis,Mahler_1932,Sprindzhuk,Sprindzhuk_2}. To our knowledge the weakest of these conditions was introduced by Kleinbock:
\begin {definition}[\cite{Kleinbock}] Let $U$ be an open set in $\RM^d$. A $C^k$-map $B:U \to \RM^e$ is called {\em $KMS$-non-degenerate (for Kleinbock-Margulis-Sprind\v{z}uk) in an affine space $\Lt$} if:
\begin{enumerate}
 \item $B(U) \subset \Lt$.
 \item For some $l \leq k$ the partial derivatives at any point of $U$ of order $l$ span the vector space associated to $\Lt$. 
\end{enumerate}
\end {definition}
It is not difficult to see that real analytic maps are always KMS non-degenerate~(see~\ref{SS::Kleinbock}) and this will constitute our starting point. A simple example of a KMS-degenerate function is given by the flat function:
$$ \RM \to \RM,\ x \mapsto e^{-1/x^2}$$

 Assuming the KMS non-degeneracy condition, one may estimate the preimage under the map $B$ of points which are badly approximated by the integer lattice $\ZM^e \subset \RM^e$ in various ways. For instance, under the KMS non-degeneracy condition, Kleinbock-Margulis and Kleinbock proved that the preimages of points satisfying an adequate Diophantine condition form a full measure set (\cite[Theorem A]{Kleinbock_Margulis}, \cite[Theorem 1.2]{Kleinbock}, see also~\cite{arithmetic}). 
  
Unfortunately the verification of non-degeneracy conditions for the frequency map of a dynamical system often rely on tedious computations and are notoriously difficult to establish. For instance, it took  many years before Féjoz and Herman proposed a proof of the non-degeneracy condition for the $N$-body problem initiated by Arnol'd in the sixties~(\cite{Arnold_trois_corps,Fejoz_Herman}, see also~\cite{ChierchiaPinzari}). In this paper, we will show that for the type of normal forms introduced by the authors in~\cite{Symplectic_torus,Versal_fields,Herman_conjecture} the KMS-non degeneracy conditions are essentially automatic. So let us now explain what type of generalisation of the usual setting is needed by considering an instructive example.
\subsection{The Poincar\'e correspondences}
\label{SS::Poincare}

In {\em Les M\'ethodes Nouvelles de la M\'ecanique C\'eleste}~\cite[§ 119]{Poincare_Methodes}. 
Poincar\'e considered the series
$$s=1+\frac{x}{1+y}+\frac{x^2}{1+2y}+\ldots=\sum_{k=0}^{\infty} \frac{x^k}{1+k y}$$
as a prototype of the kind of series that appear in perturbation theory~(see also the paper \cite{Sauzin2014} of Sauzin for the resurgence properties). In fact, similar expressions play a role in many areas of mathematics, for example in the theory of $q$-analogs of classical series (logarithm, exponential, hypergeometric), see for instance~\cite{Sondow_Zudilin}.

Assume now that we want to solve the equation
\[0=f(x,y):=-y+\frac{x}{1+y}+\frac{x^2}{1+2 y}+ +\frac{x^3}{1+3 y}+\ldots .\]

If we truncate the series at $x^n$ we obtain the rational functions
\[ f_n(x,y)=-y+\sum_{k=1}^n \frac{x^k}{1+ky} =\frac{P_n(x,y)}{D_n(y)}\]
which determines a polynomial  $P_n(x,y)$ in the variables $x$ and $y$:
\[P_1(x,y)=-(1+y)y+x,\;\; P_2(x,y)=-(1+y)(1+2y)y+(1+2y)x+x^2, \ldots\]
\[D_1(y)=1+y,\;\;\;D_2(y)=(1+y)(1+2y),\;\;\; D_3(y)=(1+y)(1+2y)(1+3y),\ldots\]
These polynomials define algebraic curves
$$V_n:=\{ (x,y) \in \CM^2:P_n(x,y)=0 \} \subset \CM \times \CM$$
of degree $(n+1)$ in $y$ and degree $n$ in $x$. These curves exhibit a peculiar bending behaviour caused by the poles, located on horizontal lines, defined by the denominators $D_n(y)$ that were cleared out. \\

Such a polynomial relation defines $y$ as a {\em multivalued function} of $x$, sometimes called a {\em correspondence}. In a neighbourhood of $x=0,y=0$, the $n$-th curve $V_n$ is the graph of a holomorphic function $b_n$ that can be expanded in power series:
\[
\begin{array}{rcl}
    b_3(x)&=&x+2x^5-14x^6+64x^7+\ldots\\
    b_4(x)&=&x+x^4-3x^5+7x^6-19x^7+\ldots\\
    b_5(x)&=&x+x^4-2x^5+x^6+12x^7+\ldots\\
    b_6(x)&=&x+x^4-2x^5+2x^6+5x^7+\ldots\\
\end{array}
\]
It is not difficult to see that $b_{n+1}(x)=b_n(x)+o(x^n)$, so these series converge coefficientwise to a formal power series solution $b(x) \in \CM[[x]]$ to the equation $f(x,y)=0$
\[ b(x)=x+x^4-2x^5+2x^6+6x^7-35x^8+86x^9+\ldots\]
We call it the {\em Birkhoff expansion} of the series $f$ and its truncation are called the {\em Birkhoff polynomials} of the series.\\

\begin{figure}[htb!]
\includegraphics[width=0.45\linewidth]{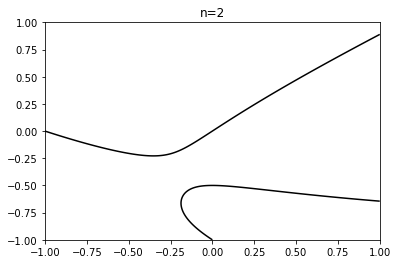}
\includegraphics[width=0.45\linewidth]{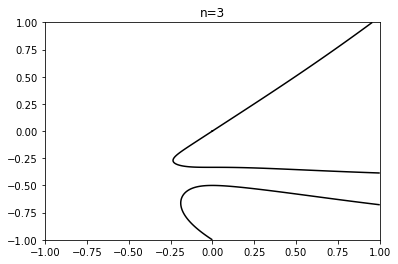}
\end{figure}
\begin{figure}[htb!]
\includegraphics[width=0.45\linewidth]{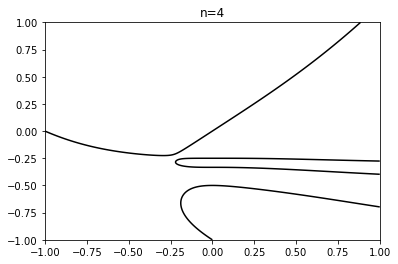}
\includegraphics[width=0.45\linewidth]{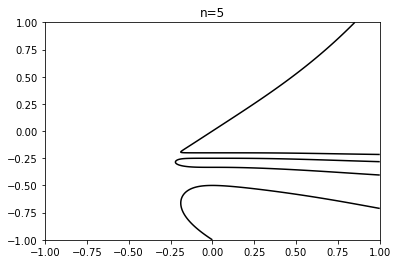}\\
{\small Correspondences in the Poincar\'e example.}
\end{figure}\ \\

The domain of definition of the $b_n$ shrinks as $n \to \infty$ due to the peculiar winding, which provides a geometrical reason for the divergence of the series $b(x)$. This is an illustration of a general divergence theorem for such expressions, that we intend to publish elsewhere~\cite{Meandromorphic}. On the other hand, the pictures strongly suggest that the part of curves in the part $y>0$ nicely converge to the graph of a $C^{\infty}$-function $B$.

We may construct such a $C^{\infty}$ function as follows: start with closed disks
$$D_{\rho}:=\{x \in \CM: |x| \le \rho\},\ D_{\eta}:=\{y \in \CM: |y| \le \eta\},\ $$
where $\rho<1$  and remove from the disk $D_{\eta}$ small disks
$D_k:=\{ y \in \CM: |y+\frac{1}{k}| < r_k \}$ around the poles $y=-1/k$.
The distance between two nearby poles being $1/k(k+1)$,
 if we choose $r_k=k^{-\a},\ \a>2$, the complement of the disks $D_k$ contains the origin. We have in the complement of $D_k$ an obvious estimate:
 $$\left|\frac{x^k}{1+ky}\right|\leq k^{\a-1}|x^k| \leq k^{\a-1}\rho^k$$
So for  $\rho <1$, we see that
 the sequence obtained by restricting the rational functions $f_n(x,y)$ to the closed set 
 $$X_{\infty} :=D_{\rho} \times \left( D_{\eta} \setminus \cup_k D_k \right) $$
converges in the $C^0$-topology to a limit function, that we denote in the same way: 
$$f: X_{\infty}  \lra \CM$$

\subsection{Whitney differentiability}
The closed set $X_{\infty}$ does contain $0$, but it is not an interior point as the poles accumulate at the origin. However, the formal power series $b$  that solves the equation $f(x,b(x))=0$ is the Taylor expansion of a function, $C^\infty$ in a sense originally discovered by Borel and now called {\em Whitney differentiability}~\cite{Borel_monogene} Let us explain this property.  

If $X \subset \RM^k$ is a locally closed subset, then a function 
$$g:X \to \RM $$ is called {\em Whitney differentiable} if  there exists a map
$$L:X \to  Hom(\RM^k,\RM)$$
such that
$$g(y)=g(x)+L(x)(y-x)+o(\| x-y\|) .$$
The main difference with usual derivatives is the requirement that the variables $x,y$ play a more
symmetric role. One defines similarly higher derivatives and Whitney $C^{\infty}$ functions. Holomorphicity in the Whitney setting is defined by requiring that the map $L$ is $\CM$-linear
($\CM^d=\RM^{2d}$) and in that case the function was called {\em monogenic} by Borel.

The Whitney extension theorem states that Whitney differentiable functions are restrictions of usual differentiable functions and the Whitney derivative commutes with restriction~\cite{Whitney_extension}. Fefferman showed that the extension operator can be chosen to be bounded linear for any finite level of differentiability~\cite{Fefferman_Whitney}. In particular, the space of functions with bounded Whitney derivatives up to order $k$ form a Banach space, like for ordinary differentiability.

Let us now show that our function $f$ are indeed Whitney differentiable on $X_\infty$. The proof is elementary but is worthwhile understanding. We will keep the notation $f_n$ for the restriction of $f_n$ to $X_\infty$.  

We first create a second sequence $X'_n$ by starting on a slightly bigger disc $D_{\rho+\e}$, $ \rho+\e<1$, and by removing from $D_{\eta}$ disks $D'_k \subset D_k$ of radius smaller than $r_k$, say for instance $r_k/2$, for $k \leq n$,  so that
$X_n \subset X'_n$. 
The Cauchy inequalities give an estimate for the derivatives of a holomorphic function on $X'_n$ restricted to $X_n$:
$$\sup_{(x,y) \in X_n}|\d_x^i\d_y^j g(x,y) | \leq \frac{2^j i! j!}{r_k^j\e^i}\sup_{(x,y) \in X'_n}| g(x,y) | $$

Using the estimate on the complement of  $D'_k$:
 $$\left|\frac{x^k}{1+ky}\right|<2 k^{\a-1}|x^k|<2 k^{\a-1}(\rho+\e)^k,$$
and the Cauchy estimates, one easily deduces  that,  for any $k \geq 0$, the sequence $(f_n)$ is a Cauchy sequence in the $C^k$-Whitney holomorphic sense  and therefore the sum function $f$ is Whitney $C^\infty$
in the holomorphic sense.

 By the Whitney extension theorem, the function $f$ is the restriction to $X_{\infty}$ of a standard $C^\infty$ function
$$F:\CM^2 \to \CM $$
The set $F=0$ serves as $C^{\infty}$-version of the limit correspondence defined by $f=0$.
As $\d_y F(0)=\d_y f(0)=1 \neq 0$, the usual $C^{\infty}$ implicit function
theorem shows that in a neighbourhood of the origin the set  $F=0$ is the
graph of a $C^\infty$ function $B$, that we call a $C^{\infty}$
{\em Birkhoff-function} attached to $f(x,y)$. Clearly, the Taylor-series of $B$ is $b$.

Somehow we proved more than expected, as the function $B$ is now also defined for $y<0$. This may look surprising, if we think of the way the real parts of the curves $V_n$ have multiple components which stretch along the horizontal lines
near the $x$-axis. The resolution of this paradox is that by removing the tubes
$D_{\rho} \times D_k$, we removed the bendings of the curve at the same time.
\subsection{Content of the paper} 
If one wants to apply the Kleinbock-Margulis theorem, one should be able to define and establish the KMS non-degeneracy condition for the correspondence before going to the Birkhoff series and doing the Whitney extension. The purpose of this paper is to show that this is indeed the case.

That the Birkhoff series $b$ does not see any of the bendings may explain
heuristically why establishing the KMS non-degeneracy condition based on considerations of $b$ alone turns out to be very difficult. Our approach consists of working as much as possible with
the {\em correspondences} lying on the product space, rather than with
the implicit functions they define. So most of the paper concerns the definition and analysis of non-degeneracy conditions for correspondences.

In the last part, we explain the application of the results of this paper to
the Herman conjecture. We also hope that our result, elementary as it may look, is of interest independently of the context in which it originally emerged,
and that the deep connection between number theory and dynamical systems can
be further strengthened, in the spirit of Siegel, Arnold and Margulis. 
 
 \section{Statement of the theorem}
 \subsection{Birkhoff functions, series and polynomials}  
 \label{SS::thetheorem}
  We consider the space $\CM^N=\CM^d \times \CM^e$ with coordinates $x=(x_1,\dots,x_d)$, $y=(y_1,\dots,y_e)$ and projections:
  
 $$\xymatrix{ & \CM^N \ar[rd]^q \ar[ld]_p & \\
 \CM^d & & \CM^e 
 }
 $$
We then consider a decreasing chain $X=(X_n)$ of locally closed subsets of $\CM^N$:
\[ X_0 \supset X_1 \supset X_2 \supset \ldots \supset X_\infty :=\cap_n X_n .\]
and assume that the sets $X_n$ are equal to the closure of their interior:
$\overline{X_n^{\circ}}=X_n$.
 We also assume that $X_n$ is of the form
 $$X_n=\Ut \times \Vt_n $$
 where $\Ut \subset \CM^d$ is a neighbourhood of the origin.

Such sequences arise often in constructions where
$X_{n+1}$ is obtained from $X_n$ by removing some open subset.
Note that the closed set $X_{\infty}$ in general could have empty interior.
We also assume to be given a formal series (in multi-index notations):
$$f(x,y)=\sum_{I \in \NM^d } a_I(y)x^I,\ a_I \in \CM\{ y \} $$
where $\CM\{ y \}$ denotes the ring of convergent power series in $y=(y_1,\dots,y_e)$.
We identify the coefficients of the series with functions of $z=(x,y)$ depending only on $y$. We use the multi-index notation and consider the Whitney $C^\a$-norms:
$$\| 	a \|^\a_U= \sup_{ z \in U \cap X_{\infty}} |\d^\a_y a(z)| , \a \in \NM^e $$
where the partial derivatives are taken only in the $y$-direction as the functions involved as coefficients of the series only depend on $y$.
 \begin{definition} We say the series $\sum_{I \in \NM^d } a_I(y)x^I$ is $X$-convergent if
  there is a neighbourhood $U$ of the origin
  such that for any $\a \in \NM^e$, the power series
  \[  \sum_{I \in \NM^d } \| a_I\|^\a_U x^I \in \CM[[x]]\]
 is convergent in $p(U)$.
\end{definition}
An  $X$-convergent series in $U \cap X_\infty$ defines a Whitney $C^\infty$ function that we denote by the same name:
$$f:U \cap X_\infty \to \CM,\ (x,y) \mapsto \sum_{I \in \NM^d } a_I(y)x^I  $$
Note that for $(x_0,y_0) \in U \cap X_\infty$, the map-germ
$$(\CM^d,x_0) \to \CM,\ x \mapsto f(x,y_0)$$
is an ordinary holomorphic map-germ.

We use the notation:
\[ V(f):=\{(x,y) \in U \cap X_\infty \;|\; f(x,y)=0\}\]
where $f=(f_1,\dots,f_k)$ are $X$-convergent series
and we let $(V(f),0)$ be the germ of this set at the origin. 
\begin{definition}
A {\em Birkhoff function-germ} associated to  $X$-convergent series $f=(f_1,\dots,f_e)$, where $\d_y f(0)$ has maximal rank, is a $C^\infty$-function germ: 
$$B:(\CM^d,0) \to (\CM^e,0) $$ whose graph (denoted $\G_B$) satisfies: 
$$(\G_B \cap X_\infty,0)= (V(f),0).$$
\end{definition}
Birkhoff polynomials and Birkhoff series are defined in a similar way by considering infinitesimal  and formal neighbourhoods instead of germs at the origin.

Note that the Taylor expansion of a Birkhoff function $B$ at the origin is the Birkhoff series of $f$. Indeed by the Whitney-Fefferman theorem~(\cite{Fefferman_Whitney}), we may choose for any $k>0$ and any neighbourhood $U$ of the origin a (real) continuous linear extension operator\footnote{We consider independently the real and imaginary parts when doing the extension.} :
$$T:C^k(U \cap X_\infty,\CM) \to C^k(U,\CM)$$
The series
$$F(x,y)=\sum_{I \in \NM^d } T(a_I(y))x^I $$
defines a $C^\infty$ function-germ.
 We have
  $$F(x,b_k(x))=o(\| x\| ^k) $$   
  where $b_k$ is the degree $k$ Birkhoff polynomial of $f$.   The manifolds $\{ y=B(x) \}$, $\{ F=0 \}$ and $\{y=b_k(x) \}$ are therefore tangent at order $k$ at the origin. Hence the degree $k$ Taylor polynomial of the function $B$ is the degree $k$ Birkhoff polynomial.
 \subsection{Non-degeneracy theorem}  
  \label{SS::thetheorem} 
 If $f=\sum_{I \in \NM^d } a_I(y)x^I$ is an $X$-convergent series then we may, by Whitney's theorem, extend $f$ to a $C^\infty$ function-germ $F$. By the implicit function theorem, the zero locus of $F$ is then the graph of a Birkhoff function-germ. So Birkhoff functions always exists. The aim of this paper is to show that the Birkhoff function can be chosen to be KMS non-degenerate: 

\begin{theorem} 
\label{T::thetheorem} Assume that the formal series:
$$f(x,y)=\sum_{I \in \NM^d } a_I(y)x^I $$
is $X$-convergent and that the matrix $\d_y f(0,0):=(\d_{y_j} f_i(0,0))$ is of maximal rank. Then
the series $f$ admits KMS non-degenerate Birkhoff function germs, i.e., there exists Whitney $C^\infty$ germs:
$$B:(\CM^d,0) \to (\CM^e,0) \text{ with } (\G_B \cap X_\infty,0)= (V(f),0). $$

\end{theorem}

It is useful to work over the field of complex numbers but of course our constructions commute with complex conjugation, in particular if the $a_I$ are real analytic and $X$ invariant under complex conjugation then the function $B$ can be chosen to be real, i.e.:
$$B(\bar x,\bar y)=\overline{B(x,y)} $$
Our construction is going to be explicit and not just an existence theorem.
 

\section{Maps, non-degeneracy and the Kleinbock lemma}
\subsection{The span}
For a subset $X \subset  \CM^e$, we denote by $[ X ]$ the {\em span} of $X$, defined as the smallest affine subspace  containing $X$ .

Clearly, if $X \subset Y$, then $[ X ] \subset [ Y ]$. 
For a {\em germ} $(X,p)$ of a set we define its span as
\[ [ X,p ]:=\bigcap_{U \in \Vt(p)} [ U \cap X ] ,\]
where $U$ runs over the set $\Vt(p)$ of open neighbourhoods of the point $p$.
We always can find a sufficiently small representative $U$ for the germ $(X,p)$,
so that  $[ X,p ]=[ U ]$.

For an open set $U \subset \CM^d$ and map $g:U \lra \CM^e$ we also define its
span as $[ g(U) ]$, and we similarly define the span of a map germ
$g:(\CM^d,\a) \to (\CM^e,\b)$ at $\a$ as the affine space
\[ [ g,\a ] :=\cap_U [ g(U),\b ] ,\]
where the intersection runs over all representatives of the germ $g$.

\begin{lemma}
\label{L::span} Let $g: U \lra \CM^e$ be an analytic map and $U(\a)$ the
  connected component of $U$ containing $\a$. Then $g(U(\alpha))$ is
  contained in the span of the germ of $g$ at $\a$:
\[ g(U(\a)) \subset [ g, \a ].\]
Consequently,
\[[ g(U(\a)) ]=[ g, \a ] .\]
\end{lemma}
\subsection{The jet-span of a map}
Given a multi-index $I \in \NM^d$, we denote by
\[ \d^Ig(\a):=\frac{\partial^I g}{\partial x^I}(\a) \in \CM^e \]
the vector of $I$-th derivative of $g$ at $\a$. We will also use the notation
$$| I|=i_1+\dots+i_d,\ I=(i_1,\dots,i_d). $$

We define the $m$-jet span of a Whitney $C^\infty$ map 
$$g:X \to \CM^e$$
at a point $\a \in X$ as the affine space:
\[ [ g,\a] _m:=[\{ \d^I g(\a): I \in \NM^d,\ |I| \leq m\} ]\]

The $m$-jet spans  form a increasing chain, so become
stable, at some minimal index $\tau=\tau(g,\a)$:
\[ [ g,\a ]_{\tau}=[ g, \a ]_{\tau+1}=\ldots \]
We call  $\tau(g,\a)$ the {\em torsion index of the map $g$ at the point $\a$}. 
The condition $ [ g,\a ]_{\tau}=[ g,\a ]$ corresponds to the KMS non-degeneracy condition.  

As an illustration, consider the germ
$$g:(\CM,0) \to (\CM^4,0):x \mapsto (1+x,x^3,x^{100},x+x^3+x^{100})=(y_1,\dots,y_4) $$
We have
\[[g,0]_0 \subset [ g,0]_1 =[ g,0]_2 \subset [ g,0]_3= [ g,0]_4= \ldots =[ g,0]_{99}\subset [ g,0]_{100}=[ g,0]_{101}=\ldots.\]
The span $[ g,0 ]$ is the affine hyperplane $[ g,0]_{100}=\{ y_4=y_1+y_2+y_3+1 \}$ and the map-germ has torsion index $\tau(g,0)=100 $ at the origin. It is KMS non-degenerate, as any holomorphic map.

A flat $C^\infty$ map-germ like
$$ g:(\CM,0) \to (\CM,0),\ x \mapsto \left\{ \begin{matrix} e^{-1/x^2} \text{ if Re } x> 0 \\
0 \text{ otherwise} \end{matrix} \right.$$
has torsion index equal to zero at the origin but it is KMS degenerate. Indeed,
  the span of this map germ is $[ g,0] =\CM$, but $[ g,0]_n=\{ 0 \}$ for any $n$ therefore:
  $$[g,0]_n \neq [ g,0 ]$$
   
\subsection{A lemma of Kleinbock}
\label{SS::Kleinbock}
The following observation due to Kleinbock was formulated originally in
the real analytic setting:

 \begin{lemma}[\cite{Kleinbock}]\label{L::Kleinbock} If $g:(\CM^d,\a) \lra (\CM^e,\b)$ is a holomorphic map germ, then it  is KMS non degenerate.
\end{lemma}
\begin{proof}
For simplicity assume that $\b=0$ so that the span is a vector space.
Let $g:U \lra \CM^e$ be a connected representative of the germ $(g,\a)$
The sequence of spans $[ g,\a ]_n$ generated by the vectors $\d^Ig(\a)$, $|I| \leq n$ stabilises at level $\tau$. Clearly $[ g,\a ]_{\tau} \subset [ g(U) ] $. To show the other inclusion, let $\ell$ be an affine linear form vanishing on $[ g,\a ]_{\tau}$. Then the Taylor series of $\ell \circ g$ at
$\a$ vanishes identically and therefore the holomorphic function $\ell \circ g$ vanishes on all of $U$. As this holds for all $\ell$ defining $[ g, \a ]_{\tau}$, we find $[ g(U)] \subset [ g, \a ]_{\tau}$, hence we have equality.
\end{proof}

As we saw previously, in the $C^\infty$ case, we may very well have a proper inclusion:
$$ \bigcup_k [ g,\a ]_k \subset [g,\a] \text{ but } \bigcup_k [ g,\a ]_k \neq [g,\a]. $$

\section{Correspondences and non-degeneracy}
\subsection{Maps and correspondences}
Many naturally appearing expressions like $x^2+y^2=1$ define $y$ only
as a multivalued function $y=\sqrt{1-x^2}$. Especially in the complex
domain, insisting on working with a monovalent function is cumbersome and
not very enlightening. Therefore, in algebraic geometry one is used to
working with {\em correspondences}, which provide a convenient extension of
the usual strict concept of function. The Poincar\'e example shows that this point of view is suited to the type of objects we are considering.  So our purpose will be to adapt the previous notions defined for maps to correspondences.


In this paper we will use the following terminology:

\begin{definition}
  A complex analytic variety $V \subset \CM^N=\CM^d \times \CM^e$ is said to
  define a {\em regular correspondence} at $\o=(\a,\b)$, if the following
  conditions
  are verified:\\
  1) the set $V$ is smooth of dimension $d$ at $\o$.\\
  2) Under the projection onto the first factor, the tangent space $T_{\o}V$
 maps isomorphically onto $\CM^d$.
\end{definition}

Of course, by the implicit function theorem, we then can find a local
parametrisation $g:(\CM^d,\a) \lra (\CM^e,\b), x \mapsto y=g(x)$
so that $(x, g(x)) \in V$, or $(\Gamma_g,\o) =(V,\o)$, where $\Gamma_g$
denotes the graph of $g$.

\subsection{The relative span}
For a subset $X \subset \CM^N:=\CM^d \times \CM^e$, we denote by $\la X\ra$ the {\em relative span} of $X$, defined as the smallest affine subspace of the form $\CM^d \times F$ containing $X$ where $F \subset \CM^e$ is an affine subspace.
Clearly, one has
\[ \la X \ra= q^{-1}(q([X])) ,\]
where $q:\CM^d \times \CM^e \lra \CM^e$ is the projection on the second factor.

Of course, the notions of span of a set and of a map germ are related:
\begin{lemma}
Let $\Gamma_g \subset \CM^d \times \CM^e=\CM^N$ denote the graph of
$g:U \lra \CM^e$ and let 
$q:\CM^N \to \CM^e$ be the projection onto the second factor. Then

\[ [ g(U)] =q(\la \Gamma_g\ra)\]  
\end{lemma}

\begin{proof} Let $\ell:  \CM^e \to \CM$ be an affine function vanishing on
$[ g(U)]$ then
$$ \ell':  \CM^d \times \CM^e \to \CM,\ (x,y) \mapsto \ell(y)$$
is an affine function vanishing on $\G_g$ and conversely.
 \end{proof}

We define the relative span at a point $\o \in X$
as the affine space
\[ \la X,\o \ra  =\cap_U \la X \cap U,\o \ra ,\]
where the intersection runs over all neighbourhoods of $\o \in \CM^N$ and get an analog of Lemma~\ref{L::span}:

\begin{lemma} Let $V \subset \CM^N$ be an analytic set and $V(\o)$ the
connected component of $V$ that contains the point $\o \in V$. Then $V(\o)$
is contained in the span of the germ $(V,\o)$:
\[ V(\o) \subset \la V, \o\ra\]
and consequently:
\[\la V(\o)\ra=\la V, \o\ra .\]
\end{lemma}

\begin{proof}
Let $\ell$ be an affine function vanishing on $\la V, \o\ra$. As a holomorphic
function vanishing on the germ $(V,\o)$, it vanishes on a neighbourhood of $V$
containing $\o$. By the identity theorem of analytic functions, $\ell$ vanishes
on the connected component $V(\o)$.
\end{proof}
\subsection{The relative jet-span of an analytic set}
For a germ of an analytic set $(V, \o)$ we define the $n$-jet span 
\[ \la V, \o \ra_n \]
 as the smallest affine subspace of the form $\CM^d \times F$ that contains the $n$-th infinitesimal
neighbourhood of the fibre at $\o \in V$ of the projection $\CM^N \to \CM^d$. It is alternatively defined by the kernel of the affine linear functions contained in the ideal $I+M_{\o,x}^{n+1}$ which depend only on the $y$-variables. Here $y_1,\dots,y_e$ denote coordinates on $\CM^e$, $I \subset \Ot_{\CM^N,p}$
is the ideal defining $(V,\o)$ and $M_{\o,x} \subset \Ot_{\CM^N,\o}$ the maximal
ideal of holomorphic functions vanishing at $x=\a$, $\o=(\a,\b)$. 

Clearly, one has\\
\[ \la V, \o \ra_1  \subset \la V, \o \ra_2 \subset  \la V, \o \ra_3 \subset \ldots, \]
so this chain stabilises for a certain index $\tau:=\tau(V,p)$, that we call the  {\em torsion index} of $V$ at $p$:
\[ \la V, \o \ra_{\tau}  = \la V, \o \ra_{\tau+1}=\ldots\]
 
\begin{lemma} For any holomorphic map-germ $g:(\CM^d,\a) \to (\CM^e,\b)$ we have
$$q(\la \G_g,\omega \ra_m)=[g,\a]_m $$
for $\o=(\a,\b)$ and any $m \in \NM$.
\end{lemma}
\begin{proof}
Write $\la \G_g,\a\ra_m=\{ \ell=0 \} $ where 
$$\ell:\CM^d \times \CM^e \to \CM^r,\ (x,y) \mapsto \ell(y)$$ is some affine map depending only on the $y$-variables. By the Taylor formula, we have
$$\ell(g(x))=\sum_{ |I| \leq m} \frac{\d^Ig(\a)}{|I|!}(x-\a)^I+o(\| x \|^m) $$
and therefore the conditions $\ell(g(x))=o(\| x \|^m) $ and $\d^Ig(\a)=0$ for $|I| \leq m$ are equivalent.
\end{proof}


\subsection{A variant of the Kleinbock lemma}

The following result is the analog of the Kleinbock lemma for analytic spaces:

\begin{lemma} 
\label{L::finite} Let $(V,\o) \subset (\CM^N,0)$ be a germ of an analytic space with torsion index $\tau$ at $\o$ then:
\[ \la V,\o \ra = \la V,\o\ra_\tau\]
\end{lemma}

\begin{proof}
  Clearly, $\la V,\o\ra_\tau \subset \la V,\o\ra$. For the converse, note that we have
  $\cap_n M_p^n =(0)$, so that a function that vanishes on all infinitesimal
  neighbourhoods of $\o$ vanishes on $(V,\o)$, from which we obtain the reverse inclusion.
 \end{proof}

\subsection{$n$-Correspondences and the finite determinacy lemma}
We now consider truncated sums as we did in the introduction for the Poincar\'e example:
 \begin{definition} A correspondence $V$ is called an {\em $n$-correspondence at $\omega$},
if it is defined by $d$ equations of the form   
\[ f(x,y) =\sum_{|I| \le n,\ I \in \NM^d} a_I(y)x^I \in \CM\{ y\}[x]\]

 

\end{definition}
 
 The following lemma is a refinement of Lemma \ref{L::finite}:

 \begin{lemma}
  \label{L::stable}
  If $V$ is an $n$-correspondence at $\omega$, then 
  $$\tau(V,\omega) \leq n $$
  that is: $\la V,\omega \ra_n =\la V, \omega\ra$.
\end{lemma}

  \begin{proof}
Consider an affine function $\ell$ vanishing on
  $\la V,\omega\ra_n$  depending only on the $y$-variables.
 Let $v \in \CM^d$ be an arbitrary vector and consider the function
  $$\g:(\CM,0) \to (\CM^e,0),\ t \mapsto f(\a+tv,y) $$
  By definition of the $n$-span,  the polynomial $\ell \circ \g$ in the variable $t$ has a zero of multiplicity $\geq n+1$ at the origin:
  $$(\ell \circ \g)^{(k)}(0)=0,\ k=0,1,2, \ldots,n.$$
  
  As $(V,\o)$ is an $n$-correspondence, the equation $\ell \circ  \g(t)=0$ has degree $\leq n$ therefore
    $\ell \circ \g $ vanishes identically. As the vector $v$ is arbitrary, we deduce that $\ell$ vanishes on $(V,\omega)$. 
  This inclusion being true for any hyperplane containing $\la V,\omega \ra_n$,
  we  get that $\la V,\o\ra \subset \la V,\omega \ra_n$.
  This proves the lemma.
  \end{proof}
\subsection{The spans stabilise}
We return to the situation of  subsection \ref{SS::thetheorem} and consider a decreasing chains of subsets of $\CM^N=\CM^d \times \CM^e$:
\[ X_0 \supset X_1 \supset X_2 \supset \ldots \supset X_\infty :=\cap_n X_n .\]
with
$\overline{X_n^{\circ}}=X_n$  
 and consider a formal power series
 \[  f(x,y) =\sum_{ I \in \NM^d} a_I(y)x^I \in \CM\{ y\} [[x]],\ f(0)=0 \]
 and define
 \[ f_n(x,y) =\sum_{|I| \le n,\ I \in \NM^d} a_I(y)x^I \in \CM\{y\}[x]\]
We denote their
zero-sets by
\[ V_n:=\{(x,y) \in X_n\;|\; f_n(x,y)=0\} \subset X_n\]

We also have to consider
\[ V_{\infty}:=\{(x,y) \in X_{\infty}\;|\; f(x,y)=0\} \subset X_{\infty}\]


\begin{lemma}One has  \[\la V_n,0\ra \subset \la V_{n+1},0\ra .\]
\end{lemma}
\begin{proof} We have
  $$\la V_n,0\ra =\la V_n,0\ra_n=\la V_{n+1},0\ra_{n} \subset \la V_{n+1},0\ra_{n+1}=\la V_{n+1},0\ra. $$
 For the first and last equalities we used that $V_n$ is an $n$-correspondence
  (for $n$ and $n+1$), the second equality is due to the fact that
  $$f_{n+1}(x,y)=f_n(x,y)+O(\|x\|^{n+1}) $$ The middle inclusion is
  by the definition of jet-span.
\end{proof}

As a consequence, the sequence $\la V_n,0\ra$ stabilises, say at level $\rho$
and we define
\[ W_{\infty} :=\la V_{n},0\ra, n \ge \rho .\]

\begin{lemma}
\label{L::limit}For $n \ge \rho$, we have:
\[ \la V_{\infty},0\ra \subset \la V_{n},0\ra= W_{\infty} \]
\end{lemma}
  \begin{proof}

    Take $(x_0,y_0) \notin W_\infty$. By definition of $W_\infty$, this means that for any $n \geq \rho$,
    the point $(x_0,y_0)$ does not belong to $ \la V_{n},0\ra$ and therefore is not in $V_n$, so we have
    $f_n(x_0,y_0) \neq 0$ for $n \ge \rho$. Consequently, for $n \ge \rho$, the holomorphic map-germs
   $$\phi_n:(\CM^d,x_0) \to \CM^e,\ x \mapsto f_{n}(x,y_0)=\sum_{|I| \leq n} a_I(y_0)x^I $$
are  non-zero at $x=x_0$.

 We now use the classical {\em Hurwitz lemma} (see e.g. \cite[Chapter 5, Theorem 2]{Ahlfors}): {\em Consider a sequence of holomorphic functions $\phi_n$ converging in the $C^0$-topology and assume that the limit $\phi_{\infty}$ is not the zero-function. Then:
  if the functions $\phi_n$ are non-zero at a point, then the limit $\phi_{\infty}$ is also non-zero at this point
  \footnote{This Hurwitz lemma, sometimes called the {\em Hurwitz theorem}, is a straightforward consequence of the Cauchy integral formula.}.}

As $f_n(x_0,y_0) \neq 0$, the coefficients $a_I$ are not all equal to zero and therefore the holomorphic map-germ
$$\phi_{\infty}:(\CM^d,x_0) \to \CM^e,\ x \mapsto  \sum_{I \in \NM^d} a_I(y_0)x^I$$
is not identically zero. From the Hurwitz lemma we conclude that  $(x_0,y_0) \notin   V_{\infty}$ and therefore  $$(V_{\infty},z_0) \subset W_{\infty},\ z_0 =(x_0,y_0)$$ 
\end{proof}
 \subsection{Proof of Theorem~\ref{T::thetheorem}}
 We keep the same notations and let 
 $$g:(\Ut \times W_\infty,0) \to (\CM^e,0) $$
 be the restriction of $f$ to $W_\infty$. By Lemma~\ref{L::limit},
 we have 
 $$(V(f),0)=(V(g),0).$$ We choose a Whitney extension $G$ of $g$. By the implicit function theorem the zero locus of $G$ is the graph of a function-germ 
 $$\b:(\CM^d,0) \to (W_\infty,0),\ x \mapsto \b(x)$$
 and we may extend the image to $\CM^e$ to get a Birkhoff function
$$B:(\CM^d,0) \to (\CM^e,0),\ x \mapsto \b(x) $$
  The Taylor expansion of $B$ at the origin is the Birkhoff series whose partial derivatives generate $W_\infty$
  by definition. Hence the function $B$ is KMS non-degenerate. This concludes the proof of the theorem.
\section{Application to the Herman conjecture}
We give a brief outlook of our proof of the Herman invariant tori conjecture, details will be published elsewhere, the reader may consult~\cite{Herman_conjecture} for further details.
\subsection{The frequency map}
In KAM theory, the invariant tori of a Hamiltonian system are parametrised by the frequencies of motions which are not very well approximated by the standard integral lattice $\ZM^e$.
In order to measure how a frequency vector $y \in \CM^e$ is approximated by the lattice, one considers the collection of resonant hyperplanes $(\S_\a)$ with 
$$\S_\a=\{ y \in \CM^e: (y,\a) =0 \},\ \a \in \ZM^e$$
where 
$$(y,\a)=\sum_{j=1}^e y_j \a_j$$
is the complexification of the Euclidean scalar product. We then fix a real positive sequence
$(\dt_\a),\ \a \in \ZM^e$, which converges to zero not too fast
and consider tubular neighbourhoods of hyperplanes
$$T_\a=\{ y \in \CM^e: |(y,\a)| < \dt_\a \},\ \a \in \ZM^e$$
 We now define the sequence $(X_n)$ by:
$$X_n=\CM^e \setminus \bigcup_{\| \a\| \leq n} T_\a $$
and this set parametrises the invariant tori, in the sense that there is a Birkhoff function
$$B: U \to  \CM^e  $$
called the {\em frequency map} whose fibres above points of $X_\infty=\bigcap_n X_n$ are (complexifications) of invariant tori. In~\cite{Symplectic_torus,Herman_conjecture}, we constructed $B$ as a map obtained from a
correspondence as we now explain.
\subsection{The small denominators ring}
We consider the polynomial ring $\CM[\omega]=\CM[\omega_1,\omega_2,\ldots,\omega_d]$ and construct a ring $SD_{\alpha}$ of {\em small denominators} at $\alpha \in \CM^d$. The ring $SD_{\alpha}$ is a subring of the field $\CM(\omega)=\CM(\omega_1,\omega_2,\ldots,\omega_d)$ of rational functions 
 defined by 
localisation of $\CM[\omega]=\CM[\omega_1,\omega_2,\ldots,\omega_d]$ with respect with the multiplicative subset $S$ generated by all linear polynomials $(\a+\omega, J), J \in \ZM^n\setminus \{0\}$:
\[SD_{\alpha}:=\CM[\omega]_{S}:=\CM[\omega,\frac{1}{(\a+\omega,J)}, J \in \ZM^n\setminus\{0\}]
\subset \CM(\omega) .\]
The elements of this ring may have poles along the resonance hyperplanes
$$H_J:=\{ \omega \in \CM^d: (\a+\omega,J)=0 \} $$
We will construct a normal form inside the algebra
$$R:=SD_{\alpha}[[\tau,p,q]] \subset \CM[[\omega,\tau, p,q]],$$
a subalgebra of the power series in $4d$ variables 
$$\omega_1,\omega_2,\ldots,\omega_d,\tau_1,\ldots,\tau_d,q_1,\ldots,q_d,p_1,\ldots,p_d .$$
We provide $R$ with the standard Poisson bracket, so that the Poisson centre of
$R$ is the ring
\[R_0:=SD_\a[[\tau]] \subset \CM[[\omega_1,\ldots,\omega_d,\tau_1,\ldots,\tau_d]].\]
We endow this ring of quasi-homogenous filtration with weights
$$w(q_i)=w(p_i)=1,\ w(\tau_i)=2,\ w(\omega_i) =0$$
We use the notations $[f]_i^j$ for the sum of terms of weighted degree $\geq i$ and $<j$ and omit the indices if $i=0$ or $j=+\infty$.
We also use the notation $O(n)$ for terms of degree $\geq n$.
\subsection{The Hamiltonian normal form iteration}
Starting from a Hamiltonian 
\[ H=\sum_{i=0}^d \alpha_i p_iq_i+O(3),\]
we first form the $\omega$-extension of $H$:
\[ F_0=\sum_{i=0}^d (\alpha_i+\omega_i) p_iq_i+O(3),\]

Then we solve a {\em homological equation} of the form:
$$v_0(A_0)=[F_0]_3^4 +t_0,\;\;\; t_0 \in  R_0+ I^2 ,$$ 
where $I$ is the ideal generated by the  $p_iq_i-\tau_i$ and $v_0$ is a Poisson derivation.
The definition might seem involved but it produces a sequence of biholomorphic maps $e^{-v_i}$ which are defined globally and not "blocked by resonances."

Exponentiating the Poisson derivation $v_0$, we produce a Poisson automorphism  $e^{-v_0}$. The application of $e^{-v_0}$ transforms $F_0$ into $F_1$, where the cubic term is  removed modulo an element of $R_0+I^2$ ; we put $A_1=A_0+t_0$. For the $n=0$ case, it turns out that $t=0$ but at the next level we have to solve
$$v_1(A_1)=[F_1]_4^6 + t_1 ,\;\;\; t_1 \in R_0+I^2$$ 
for the degree $4$ and $5$ part of $F_1$ on $A_1$ and, as a general rule $t_1 \neq 0$. 

Then the application of $e^{-v_1}$ to $F_1$ produces $F_2$, where now these terms of degree $4$ and $5$ are removed, but certain terms in $R_0+I^2$ are introduced. These remaining terms we add to $A_1$ and obtain $A_2$. Next we solve the homological equation for the terms of degree $6,7,8,9$ of $F_2$, but now on  $A_2$, etc. Thus we obtain a by iteration a sequence of triples
\[ (F_n, A_n, v_n ), \;\;\;n=0,1,2,\ldots\]
We have correspondences defined by the rational functions 
$$[e^{-v_n} \dots e^{-v_0} \omega_i]^{2^n}= \sum_{\a, |\a| \le 2^n} a_\a(\omega) \tau^\a,\ a_\a \in SD_\a $$
 precisely of the kind envisioned by Poincar\'e!  This power series is $X$-convergent, the partial sums define complex variety and the limit variety
 parametrises the pairs consisting of invariant tori and associated frequencies. 
 
 The theorem of this paper shows that the limiting correspondence contains the graph of a KMS
 non-degenerate map.  Using an adequate variant of the Kleinbock-Margulis theorem~(\cite{arithmetic}), one deduces the fact that the invariant tori form a set of positive measure.  The convergence
 of the iteration scheme proposed  is subtle, but it turns out that it can be
 streamlined using the theory of Banach functors developed in~\cite{Functors}.  In that
 set-up, the non-degeneracy argument is easy and, in fact, as this paper shows, ``almost trivial''. 
 
\bibliographystyle{amsplain}
\bibliography{master}
\end{document}